\documentclass{article}
\usepackage{latexsym,amsfonts,amsmath,graphics}
\usepackage{epsfig}
\usepackage{enumitem}
\usepackage{amssymb,mathrsfs}

\newtheorem{theorem}{Theorem}
\newtheorem{lemma}{Lemma}

\newtheorem{definition}{Definition}

\newenvironment{proof}{\begin{trivlist} \item[\hskip\labelsep{\it Proof.}]}{$\hfill\Box$\end{trivlist}}

\newcommand{\rd}{\,\mathrm{d}}
\newcommand{\bsj}{\boldsymbol{j}}
\newcommand{\bsx}{\boldsymbol{x}}

\newcommand{\bsm}{\boldsymbol{m}}
\newcommand{\bst}{\boldsymbol{t}}

\newcommand{\bssigma}{\boldsymbol{\sigma}}

\newcommand{\bszero}{\boldsymbol{0}}

\newcommand{\RR}{\mathbb{R}}

\newcommand{\LL}{{\cal L}}

\newcommand{\FF}{\mathbb{F}}
\newcommand{\NN}{\mathbb{N}}

\newcommand{\cP}{\mathcal{P}}

\newcommand{\qed} {\hfill \Box \vspace{0.5cm}} 
\allowdisplaybreaks

\title{Optimal order of $L_p$-discrepancy of digit shifted Hammersley point sets in dimension 2}
\author{Aicke Hinrichs, Ralph Kritzinger and Friedrich Pillichshammer\thanks{R. Kritzinger and F. Pillichshammer are supported by the Austrian Science Fund (FWF): Project F5509-N26, which is a part of the Special Research Program "Quasi-Monte Carlo Methods: Theory and Applications".}}
\date{}

\begin{document}

\maketitle

\begin{abstract}
It is well known that the two-dimensional Hammersley point set consisting of $N=2^n$ elements (also known as Roth net) does not have optimal order of $L_p$-discrepancy for $p \in (1,\infty)$ in the sense of the lower bounds according to Roth (for $p \in [2,\infty)$) and Schmidt (for $p \in (1,2)$). On the other hand, it is also known that slight modifications of the Hammersley point set can lead to the optimal order $\sqrt{\log N}/N$ of $L_2$-discrepancy, where $N$ is the number of points. Among these are for example digit shifts or the symmetrization. In this paper we show that these modified Hammersley point sets also achieve optimal order of $L_p$-discrepancy for all $p \in (1,\infty)$.  
\end{abstract}

\centerline{\begin{minipage}[hc]{110mm}{
}
\end{minipage}}

\section{Introduction}

For a finite set $\cP_{N,s} =\{\bsx_0,\ldots
,\bsx_{N-1}\}$ of points in the $s$-dimensional unit-cube $[0,1)^s$
the {\it local discrepancy} is defined as
$$D_N(\cP_{N,s},\bst)=\frac{A_N([\bszero,\bst),\cP_{N,s})}{N}- t_1 t_2 \cdots t_s,$$ where $\bst = (t_1,t_2,\ldots, t_s) \in [0,1]^s$ and
$A_N([\bszero,\bst), \cP_{N,s})$ denotes the number of indices $k$ with
$\bsx_k \in [0,t_1)\times \ldots \times [0,t_s) =: [\bszero, \bst)$. The local discrepancy measures the difference of the portion of points in an axis parallel box containing the origin and the volume of this box. Hence it is a measure of the irregularity of distribution of a point set in $[0,1)^s$.

\begin{definition}
Let $p \in [1,\infty]$. The {\it $L_p$-discrepancy} of $\cP_{N,s}$ is defined as the $L_p$-norm of the local discrepancy
\begin{eqnarray}\label{Lq-definition}
L_{p,N}(\cP_{N,s})=\| D_N(\cP_{N,s},\cdot)\|_{L_p}= \left(\int_{[0,1]^{s}} |D_N(\cP_{N,s},\bst)|^p \rd\bst\right)^{1/p}
\end{eqnarray}
with the obvious modifications for $p=\infty$. 
\end{definition}

The $L_p$-discrepancy can also be linked to the integration error of a quasi-Monte Carlo rule, see, e.g. \cite{DP10,Nie73,SloWo} for the error in the worst-case setting and \cite{Woz} for the average case setting.

One of the questions on irregularities of distribution is concerned with the precise order of convergence of the smallest possible values of $L_p$-discrepancy as $N$ goes to infinity.

In this paper we only deal with the case $s=2$ and $p \in (1,\infty)$ and consider modifications of the {\it two-dimensional Hammersley point set} with $N=2^n$ elements (also known as {\it Roth net}) given by  
\begin{equation}\label{defHam}
{\cal R}_n=\Big\{ \Big( \frac{t_n}{2}+\frac{t_{n-1}}{2^2}+\cdots + \frac{t_1}{2^n} ,  
                       \frac{t_{1}}{2}+\frac{t_{2}}{2^2}  +\cdots + \frac{t_n}{2^n} \Big) \ : \ 
                       t_1,\ldots,t_n \in \{0,1\}   \Big\}.
\end{equation}
It is well known (see, for example, \cite[Corollary~1]{P2002}) that for all $p \in [1,\infty)$ we have $$\lim_{N \rightarrow \infty} \frac{N L_{p,N}({\cal R}_n)}{\log N}=\frac{1}{8 \log 2},$$ where here and throughout the paper $\log$ denotes the natural logarithm. Hence the two-dimensional Hammersley point set does not have optimal order of $L_p$-discrepancy with respect to the general lower bound by Roth (for $p \in [2,\infty)$) and Schmidt (for $p \in (1,2)$), see Section~\ref{sec2}.

In this paper we consider digit shifted Hammersley point sets (Section~\ref{sec4}) and symmetrized digit shifted Hammersley point sets (Section~\ref{sec5}) and show that for both cases we can achieve the optimal order of $L_p$-discrepancy for all $p \in (1,\infty)$ (see Theorem~\ref{thm1}, \ref{thm2} and \ref{thm3}). The optimality of the
digit shifted Hammersley point sets for the $L_p$-discrepancy was recently shown by Markhasin \cite{Lev13d,Lev13}. The proof there is indirect via optimality of the norm of
the discrepancy function in Besov spaces with dominating mixed smoothness together with embedding theorems between Besov spaces and Triebel-Lizorkin spaces
which contain $L_p$-spaces as special cases. The main tool there is the computation of Haar coefficients of the discrepancy function which, for the
digit shifted Hammersley point sets, were already computed in \cite{hin2010}. We give a direct proof here via Littlewood-Paley theory which is accessible without
knowledge of function space theory. 

\paragraph{Notation.} Throughout the paper we use the following notation. For functions $f,g:\NN \rightarrow \RR^+$ we write $g(N) \ll f(N)$ (or $g(N) \gg f(N)$), 
if there exists a $C>0$ such that $g(N) \le C f(N)$ (or $g(N) \ge C f(N)$) for all $N \in \NN$, $N \ge 2$. If we would like to stress that the quantity $C$ may also depend on other variables than $N$, say $\alpha_1,\ldots,\alpha_w$, 
this will be indicated by writing $\ll_{\alpha_1,\ldots,\alpha_w}$ (or $\gg_{\alpha_1,\ldots,\alpha_w}$). Sometimes we also use $f(N) \asymp g(N)$ which means that $f(N) \ll g(N)$ and $f(N) \gg g(N)$ simultaneously.\\

Before we continue we survey some known results from discrepancy theory:

\section{A brief survey of known results}\label{sec2}

In 1954 Roth~\cite{Roth} proved that for every $N$-element point set $\cP_{N,s}$ in $[0,1)^s$ we have 
\begin{equation}\label{bd_roth}
L_{2,N}(\cP_{N,s}) \gg_s \frac{(\log N)^{\frac{s-1}{2}}}{N}.
\end{equation} 
Roth's original proof can be found in \cite{Roth}. Further proofs are presented in \cite{BC,DP10,DT97,HM,kuinie,mat}. According to  a result of Hinrichs and Markhasin \cite{HM} the implied constant $c_s$ can be chosen as $$c_s= \frac{7}{27 \cdot 2^{2s-1} (\log 2)^{(s-1)/2} \sqrt{(s-1)!}}.$$ From the monotonicity of the $L_p$-norm it is evident that Roth's lower bound \eqref{bd_roth} also holds for the $L_p$-discrepancy for any $p \in [2,\infty)$. Furthermore, it was shown by Schmidt~\cite{schX} that also for any $p \in (1,2)$ we have
\begin{equation}\label{extsch}
L_{p,N}(\cP_{N,s}) \gg_{s,p} \frac{(\log N)^{\frac{s-1}{2}}}{N}
\end{equation}
for any $N$-element point set $\cP_{N,s}$ in $[0,1)^s$.

In 1956 Davenport \cite{dav} proved that the lower bound \eqref{bd_roth} is best possible for the $L_2$-discrepancy in dimension 2. He considered the $N=2 M$ points $(\{\pm n \alpha\},n/M)$ for $n=1,2,\ldots,M$ and showed that if $\alpha$ is an irrational number having a continued fraction expansion with bounded partial quotients then the $L_2$-discrepancy of the collection $\cP_{N,2}^{{\rm sym}}(\alpha)$ of these points satisfies $$L_{2,N}(\cP_{N,2}^{{\rm sym}}(\alpha)) \ll_{\alpha} \frac{\sqrt{\log N}}{N}$$ where the implied constant only depends on $\alpha$. Nowadays there exist several variants of such ``symmetrized'' point sets having optimal order of $L_2$-discrepancy in dimension 2, see, for example, the work of Larcher and Pillichshammer \cite{lp} who study the $L_2$-discrepancy of symmetrized digital nets or the work of Proinov \cite{pro1988a}. A nice discussion of the topic, which is often referred to as {\it Davenport's reflection principle} can be found in \cite{CS2000}. Symmetrized Hammersley point sets will be considered in Section~\ref{sec5}. Recently Bilyk \cite{Bil} proved that unsymmetrized versions of Davenport's point sets, i.e., point sets of the form $\cP_{N,2}(\alpha)=\{(\{n \alpha\},n/N)\ : \, n=0,1,\ldots,N-1\}$ satisfy $$L_{2,N}(\cP_{N,2}(\alpha)) \ll_{\alpha} \frac{\sqrt{\log N}}{N}$$ if and only if the bounded partial quotients of $\alpha=[a_0;a_1,a_2,\ldots]$ satisfy $$\left|\sum_{k=0}^n(-1)^k a_k\right| \ll_\alpha \sqrt{n}.$$ Further examples of two-dimensional finite point sets with optimal order of $L_2$-discrepancy which are based on scrambled digital nets can be found in \cite{FauPi09a,FauPi09,FauPiPri11,FauPiPriSch09,goda,HZ,KriPi2006}. One prominent instance in this class are digit shifted Hammersley point sets. For these it is well known that the $L_2$-discrepancy is of optimal order if the number of $0$s and $1$s in the dyadic shifts are balanced (see, for example, \cite{KriPi2006}). Digit shifted Hammersley point sets will be considered in Section~\ref{sec4}\\

For completeness we mention also some results for arbitrary dimensions: in \cite{roth2} Roth proved that the bound \eqref{bd_roth} is best possible in dimension 3 and finally Roth \cite{Roth4} and Frolov \cite{Frolov} proved that the bound \eqref{bd_roth} is best possible in any dimension. In \cite{chen1980} Chen showed that the $L_p$-discrepancy bound \eqref{extsch} is best possible in the order of magnitude in $N$ for any $p \in (1,\infty)$, i.e., for every $N,s \in \NN$, $N \ge 2$, there exists an $N$-element point set $\cP_{N,s}$ in $[0,1)^s$ such that $$L_{p,N}(\cP_{N,s}) \ll_{s,p} \frac{(\log N)^{\frac{s-1}{2}}}{N},$$ where the implied constant only depends on $s$ and $p$, but not on $N$. See also \cite{BC} for more information. Further existence results for point sets with optimal order of $L_p$-discrepancy can be found in \cite{CS3,DP05b,Skr12}. However, all these results for dimension 3 and higher are only existence results obtained by averaging arguments and it remained a long standing open question in discrepancy theory to find explicit constructions of finite point sets with optimal order of $L_2$-discrepancy in the sense of Roth's lower bound. The breakthrough in this direction was achieved by Chen and Skriganov \cite{CS02}, who proved a complete solution to this problem. They gave for the first time for every integer $N \ge 2$ and every dimension $s \in \NN$, explicit constructions of finite $N$-element point sets in $[0,1)^s$ whose $L_2$-discrepancy achieves an order of convergence of $(\log N)^{(s-1)/2}/N$. The result in \cite{CS02} was extended to the $L_p$-discrepancy for $p \in (1,\infty)$ by Skriganov~\cite{Skr} who used Littlewood-Paley theory in his proofs. This will also play a major role in our paper, see Lemma~\ref{lpi}. Further constructions of point sets with optimal $L_p$-discrepancy can be found in \cite{D14,DP14,Lev14}. See also \cite{bil11,DP14a} for more detailed surveys.\\

\section{$L_p$-discrepancy of digit shifted Hammersley point sets}\label{sec4}

Let $\bssigma=(\sigma_1,\sigma_2,\ldots,\sigma_{n}) \in \{0,1\}^n$. The {\it two-dimensional digit shifted Hammersley point set} is given by
$$
 {\cal R}_{n,\bssigma} = \Big\{ \Big( \frac{t_n}{2}+\frac{t_{n-1}}{2^2}+\cdots + \frac{t_1}{2^n} ,  
                       \frac{t_{1} \oplus \sigma_1}{2}+\frac{t_{2}\oplus \sigma_2}{2^2}  +\cdots + \frac{t_n \oplus \sigma_n}{2^n} \Big) \ : \ 
                       t_1,\ldots,t_n \in \{0,1\}   \Big\} 
$$
where $t \oplus \sigma =t+\sigma \pmod{2}$ for $t,\sigma \in \{0,1\}$. 
This point set contains $N=2^n$ elements. If $\bssigma=\bszero=(0,0,\ldots,0)$, then we obtain the classical two-dimensional Hammersley point set \eqref{defHam}.

It was shown in \cite{HZ} that two-dimensional digit shifted Hammersley point sets satisfy the $L_2$-discrepancy estimate
$$L_{2,N}({\cal R}_{n,\bssigma}) \ll  \frac{\sqrt{\log N}}{N},$$
whenever $\bssigma=(1,0,1,0,\ldots)$, which is optimal according to \eqref{bd_roth}. An exact formula for $L_{2,N}({\cal R}_{n,\bssigma})$ and a generalization of the result can be found in \cite{KriPi2006}. In  particular it is shown in \cite{KriPi2006} that $$L_{2,N}({\cal R}_{n,\bssigma}) \ll  \frac{\sqrt{\log N}}{N},$$
whenever the number of $0$- and $1$-components of $\bssigma$ are ``more or less'' balanced, i.e., $\#\{j \ : \ \sigma_j=0\} \approx n/2$. See also \cite[Section~3.1]{DHP14}. Motivated by these results the question\footnote{This question was stated by J. Dick during a private communication at the Oberwolfach workshop ``Uniform Distribution Theory and Applications'', Sept. 25 -- Oct. 5, 2013.} arises whether digit shifted Hammersley point sets can also achieve optimal order of $L_p$-discrepancy for any $p \in (1,\infty)$?

We answer this question in the affirmative and generalize the results in \cite{HZ,KriPi2006} to the case of $L_p$-discrepancy for all $p \in (1,\infty)$. The following result is already announced in \cite{DHP14}.

\begin{theorem}\label{thm1}
Let $p \in (1,\infty)$. Let $\bssigma \in \{0,1\}^n$ and $a_n=\#\{j \ : \ \sigma_j=0\}$. The $L_p$-discrepancy of the two-dimensional digit shifted Hammersley point set satisfies $$L_{p,N}({\cal R}_{n,\bssigma}) \ll_p \frac{\sqrt{\log N}}{N}$$ if and only if  $|2 a_n-n| \ll_p \sqrt{n}$.
\end{theorem}

In other words, we achieve exactly for those shifts $\bssigma$ optimal order of $L_p$-discrepancy of ${\cal R}_{n,\bssigma}$ for which the number of $0$- and $1$-components of $\bssigma$ are ``more or less'' balanced.

The proof of this result uses the Haar system (in base $2$) which we introduce now: 

To begin with, a {\it dyadic interval} of length $2^{-j}, j\in {\mathbb N}_0,$ in $[0,1)$ is an interval of the form 
$$ I=I_{j,m}:=\left[\frac{m}{2^j},\frac{m+1}{2^j}\right) \ \ \mbox{for } \  m=0,1,\ldots,2^j-1.$$ We also define $I_{-1,0}=[0,1)$.
The left and right half of $I=I_{j,m}$ are the dyadic intervals $I^+ = I_{j,m}^+ =I_{j+1,2m}$ and $I^- = I_{j,m}^- =I_{j+1,2m+1}$, respectively. The {\it Haar function} $h_I = h_{j,m}$ with support $I$ 
is the function on $[0,1)$ which is  $+1$ on the left half of $I$, $-1$ on the right half of $I$ and 0 outside of $I$. The $L_\infty$-normalized {\it Haar system} consists of
all Haar functions $h_{j,m}$ with $j\in{\mathbb N}_0$ and  $m=0,1,\ldots,2^j-1$ together with the indicator function $h_{-1,0}$ of $[0,1)$.
Normalized in $L_2([0,1))$ we obtain the {\it orthonormal Haar basis} of $L_2([0,1))$. 

Let ${\mathbb N}_{-1}=\{-1,0,1,2,\ldots\}$ and define ${\mathbb D}_j=\{0,1,\ldots,2^j-1\}$ for $j\in{\mathbb N}_0$ and ${\mathbb D}_{-1}=\{0\}$.
For $\bsj=(j_1,\dots,j_s)\in{\mathbb N}_{-1}^s$ and $\bsm=(m_1,\dots,m_s)\in {\mathbb D}_{\bsj} :={\mathbb D}_{j_1}\times \ldots \times {\mathbb D}_{j_s}$, 
the {\it Haar function} $h_{\bsj,\bsm}$
is given as the tensor product 
$$h_{\bsj,\bsm} (x) = h_{j_1,m_1}(x_1)\, \cdots \, h_{j_s,m_s}(x_s) \ \ \ \mbox{ for } \bsx=(x_1,\dots,x_s)\in[0,1)^s.$$
The boxes $$I_{\bsj,\bsm} = I_{j_1,m_1} \times \ldots \times I_{j_s,m_s}$$ are called {\it dyadic boxes}.
Two boxes $I_{\bsj_1,\bsm_1}$ and $I_{\bsj_2,\bsm_2}$ have the {\it same shape} if $\bsj_1=\bsj_2$.
A crucial combinatorial property is that for $\bsj=(j_1,\dots,j_s)\in{\mathbb N}_{0}^s$, there are exactly $2^{j_1+\dots+j_s}$ boxes of that shape which are mutually disjoint.

The $L_\infty$-normalized tensor {\it Haar system} consists of all Haar functions $h_{\bsj,\bsm}$ with $\bsj\in{\mathbb N}_{-1}^s$ and  
$\bsm \in {\mathbb D}_{\bsj}$. Normalized in $L_2([0,1)^s)$ we obtain the {\it orthonormal Haar basis} of $L_2([0,1)^s)$.\\

Direct, but in some cases a little tedious computations, for which we refer to \cite[Theorem~3.1]{hin2010}, give the {\it Haar coefficients} $$\mu_{\bsj,\bsm} = \langle  D_N({\cal R}_{n,\bssigma}, \, \cdot \,), h_{\bsj,\bsm} \rangle=\int_{[0,1]^2} D_N({\cal R}_{n,\bssigma},\bst)h_{\bsj,\bsm}(\bst) \rd \bst$$ of the local discrepancy of the two-dimensional digit shifted Hammersley point sets:
\begin{lemma}[{\cite[Theorem~3.1]{hin2010}}] \label{lem:HCH}
Let $\bsj=(j_1,j_2)\in \NN_{0}^2$. Then
 \begin{itemize}
  \item[(i)] if $j_1+j_2<n-1$ and $j_1,j_2\ge 0$ then $|\mu_{\bsj,\bsm}| = 2^{-2(n+1)}$.
  \item[(ii)] if $j_1+j_2\ge n-1$ and $0\le j_1,j_2\le n$ then $|\mu_{\bsj,\bsm}| \le 2^{-(n+j_1+j_2+1)}$ and
     $|\mu_{\bsj,\bsm}| = 2^{-2(j_1+j_2+2)}$ for all but at most $2^n$ coefficients $\mu_{\bsj,\bsm}$ with $\bsm\in {\mathbb D}_{\bsj}$ (the latter appears if there is no point of ${\cal R}_{n,\bssigma}$ in the interior of $I_{\bsj,\bsm}$).
  \item[(iii)] if $j_1 \ge n$ or $j_2 \ge n$ then $|\mu_{\bsj,\bsm}| = 2^{-2(j_1+j_2+2)}$.
 \end{itemize} 
 
 Now let $\bsj=(-1,k)$ or $\bsj=(k,-1)$ with $k\in \NN_0$. Then
 \begin{itemize}
  \item[(iv)] if $k<n$ then $|\mu_{\bsj,\bsm}| \le 2^{-(n+k)}$.
  \item[(v)] if $k\ge n$  then $|\mu_{\bsj,\bsm}| = 2^{-(2k+3)}$.
 \end{itemize}
 Finally, if $a_n=\#\{j \ : \ \sigma_j=0\}$ then
  \begin{itemize}
   \item[(vi)] $\mu_{(-1,-1),(0,0)} = 2^{-(n+3)}  (2 a_n +4 -n)+ 2^{-2(n+1)}$.
  \end{itemize} 
\end{lemma}

In the proof of Theorem~\ref{thm1} we make use of these results in conjunction with the {\it Littlewood-Paley inequality} which provides a tool which can be used to replace Parseval's equality and Bessel's inequality for functions in $L_p(\RR^s)$ with $p \in (1,\infty)$.
It involves the {\em square function} $S(f)$ of a function $f\in L_p([0,1)^2)$ (we restrict ourselves to the case $s=2$ since this is the only case of interest here) which is given as
$$ S(f) = \left( \sum_{\bsj \in \NN_{-1}^2} \sum_{\bsm \in \mathbb{D}_{\bsj}} 2^{2|\bsj|} \, \langle f , h_{\bsj,\bsm} \rangle^2 \, {\mathbf 1}_{I_{\bsj,\bsm}} \right)^{1/2},$$ where for $\bsj=(j_1,j_2)$ we write $|\bsj|=\max\{0,j_1\}+\max\{0,j_2\}$,  and where ${\mathbf 1}_I$ is the characteristic function of $I$.

\begin{lemma}[Littlewood-Paley inequality]\label{lpi}
 Let $p \in (1,\infty)$ and let $f\in L_p([0,1)^2)$. Then 
 $$ \| S(f) \|_{L_p} \asymp_{p} \| f \|_{L_p}.$$
\end{lemma}

Proofs of this equivalence of norms between the function and its square function and further details also yielding the right asymptotic behavior of the involved constants can be found in \cite{bur88,pip86,ste93,WG91}.\\

\noindent{\it Proof of Theorem~\ref{thm1}.} First we show the sufficiency of the condition. Using Lemma~\ref{lpi} with $f=D_{N}({\cal R}_{n,\bssigma},\cdot)$ we have
\begin{eqnarray*}
L_{p,N}({\cal R}_{n,\bssigma}) & = & \|D_{N}({\cal R}_{n,\bssigma},\cdot)\|_{L_p}\\
& \ll_p & \|S(D_{N}({\cal R}_{n,\bssigma},\cdot))\|_{L_p}\\
& = & \left\| \left( \sum_{\bsj \in \NN_{-1}^2} \sum_{\bsm \in \mathbb{D}_{\bsj}} 2^{2|\bsj|} \, \mu_{\bsj,\bsm}^2 \, {\mathbf 1}_{I_{\bsj,\bsm}} \right)^{1/2} \right\|_{L_p}\\
& = & \left\|\sum_{\bsj \in \NN_{-1}^2} 2^{2|\bsj|} \sum_{\bsm \in \mathbb{D}_{\bsj}}  \mu_{\bsj,\bsm}^2 \, {\mathbf 1}_{I_{\bsj,\bsm}} \right\|_{L_{p/2}}^{1/2}\\
& \le & \left( \sum_{\bsj \in \NN_{-1}^2} 2^{2|\bsj|} \left\| \sum_{\bsm \in \mathbb{D}_{\bsj}}  \mu_{\bsj,\bsm}^2 \, {\mathbf 1}_{I_{\bsj,\bsm}} \right\|_{L_{p/2}} \right)^{1/2},
\end{eqnarray*}
where we used Minkowski's inequality for the $L_{p/2}$-norm. Hence, in order to prove the result it suffices to show that 
\begin{equation}\label{eq2do}
\sum_{\bsj \in \NN_{-1}^2} 2^{2|\bsj|} \left\| \sum_{\bsm \in \mathbb{D}_{\bsj}}  \mu_{\bsj,\bsm}^2 \, {\mathbf 1}_{I_{\bsj,\bsm}} \right\|_{L_{p/2}} \ll \frac{n}{2^{2n}}.
\end{equation}
To this end we split the sum over the $\bsj$'s into several parts and apply Lemma~\ref{lem:HCH}:
\begin{itemize}
\item $\bsj \in \NN_{0}^2$ such that $|\bsj| < n-1$: According to {\it (i)} of Lemma~\ref{lem:HCH} we have
\begin{eqnarray*}
\sum_{\bsj \in \NN_{0}^2 \atop |\bsj| < n-1} 2^{2|\bsj|} \left\| \sum_{\bsm \in \mathbb{D}_{\bsj}}  \mu_{\bsj,\bsm}^2 \, {\mathbf 1}_{I_{\bsj,\bsm}} \right\|_{L_{p/2}} & = & \sum_{\bsj \in \NN_{0}^2 \atop |\bsj| < n-1} 2^{2|\bsj|} 2^{-4(n+1)} \left\| \sum_{\bsm \in \mathbb{D}_{\bsj}}  {\mathbf 1}_{I_{\bsj,\bsm}} \right\|_{L_{p/2}}\\
& = & \sum_{\bsj \in \NN_{0}^2 \atop |\bsj| < n-1} 2^{2|\bsj|} 2^{-4(n+1)}\\
& \ll & \frac{1}{2^{4n}} \sum_{k=0}^{n-2} 2^{2 k} \underbrace{\sum_{j_1,j_2=0 \atop j_1+j_2=k}^{\infty} 1}_{= k+1 \le n -1}\\
& \ll & \frac{n}{2^{2n}}. 
\end{eqnarray*}
Here we used that for fixed $\bsj$ the intervals $I_{\bsj,\bsm}$ with $\bsm \in \mathbb{D}_{\bsj}$ form a partition of the unit-square $[0,1)^2$ and hence $\sum_{\bsm \in \mathbb{D}_{\bsj}}  {\mathbf 1}_{I_{\bsj,\bsm}}=1$. 
\item $|\bsj| \ge n-1$ and $0 \le j_1,j_2 \le n$: Let $I_{\bsj,\bsm}^{\circ}$ denote the interior of a dyadic box $I_{\bsj,\bsm}$. According to {\it (ii)} of Lemma~\ref{lem:HCH} we have
\begin{eqnarray*}
\lefteqn{\sum_{j_1,j_2=0 \atop |\bsj| \ge n-1}^n 2^{2|\bsj|} \left\| \sum_{\bsm \in \mathbb{D}_{\bsj}}  \mu_{\bsj,\bsm}^2 \, {\mathbf 1}_{I_{\bsj,\bsm}} \right\|_{L_{p/2}}}\\
& = & \sum_{j_1,j_2=0 \atop |\bsj| \ge n-1}^n 2^{2|\bsj|} \left\| \sum_{\bsm \in \mathbb{D}_{\bsj}\atop {\cal R}_{n,\bssigma} \cap I_{\bsj,\bsm}^{\circ} =\emptyset}  \mu_{\bsj,\bsm}^2 \, {\mathbf 1}_{I_{\bsj,\bsm}} + \sum_{\bsm \in \mathbb{D}_{\bsj}\atop {\cal R}_{n,\bssigma} \cap I_{\bsj,\bsm}^{\circ} \not=\emptyset}  \mu_{\bsj,\bsm}^2 \, {\mathbf 1}_{I_{\bsj,\bsm}} \right\|_{L_{p/2}}\\
& \le & \sum_{j_1,j_2=0 \atop |\bsj| \ge n-1}^n 2^{2|\bsj|} 2^{-4(|\bsj|+2)} + \sum_{j_1,j_2=0 \atop |\bsj| \ge n-1}^n 2^{2|\bsj|} 2^{-2(n+|\bsj|+1)} \left\| \sum_{\bsm \in \mathbb{D}_{\bsj}\atop {\cal R}_{n,\bssigma} \cap I_{\bsj,\bsm}^{\circ} \not=\emptyset} {\mathbf 1}_{I_{\bsj,\bsm}} \right\|_{L_{p/2}},
\end{eqnarray*}
where we used Minkowski's inequality again. For the first sum in this estimate we have
\begin{eqnarray*}
\sum_{j_1,j_2=0 \atop |\bsj| \ge n-1}^n 2^{2|\bsj|} 2^{-4(|\bsj|+2)} & \ll & \sum_{k=n-1}^{2 n} \frac{1}{2^{2k}} \sum_{j_1,j_2=0 \atop j_1+j_2=k}^n 1 \\
& = &  \sum_{k=n-1}^{2 n} \frac{1}{2^{2k}} \sum_{j_1=0 \atop 0 \le k-j_1\le n}^n 1 \\
& \ll & \frac{n}{2^{2n}}.  
\end{eqnarray*}
Now we turn to the second sum 
\begin{equation}\label{su2nd}
\sum_{j_1,j_2=0 \atop |\bsj| \ge n-1}^n 2^{2|\bsj|} 2^{-2(n+|\bsj|+1)} \left\| \sum_{\bsm \in \mathbb{D}_{\bsj}\atop {\cal R}_{n,\bssigma} \cap I_{\bsj,\bsm}^{\circ} \not=\emptyset} {\mathbf 1}_{I_{\bsj,\bsm}} \right\|_{L_{p/2}}.
\end{equation}
Note that $$\sum_{\bsm \in \mathbb{D}_{\bsj}\atop {\cal R}_{n,\bssigma} \cap I_{\bsj,\bsm}^{\circ} \not=\emptyset} {\mathbf 1}_{I_{\bsj,\bsm}}$$ is the indicator function of a set, say $A_{\bsj}$, of measure at most $2^{n-|\bsj|}$. Hence \eqref{su2nd} can be written as 
\begin{eqnarray*}
\frac{1}{2^{2(n+1)}} \sum_{j_1,j_2=0 \atop |\bsj| \ge n-1}^n \left\|{\mathbf 1}_{A_{\bsj}} \right\|_{L_{p/2}} & = & \frac{1}{2^{2(n+1)}} \sum_{j_1,j_2=0 \atop |\bsj| \ge n-1}^n \left(\int_{[0,1]^2} {\mathbf 1}_{A_{\bsj}}(\bsx) \rd \bsx\right)^{2/p}\\
& \ll & \frac{1}{2^{2n}} \sum_{j_1,j_2=0 \atop |\bsj| \ge n-1}^n (2^{n-|\bsj|})^{2/p}\\
& = & \frac{1}{2^{2n}} 2^{2 n/p} \sum_{k=n-1}^{2 n} \frac{1}{2^{2 k/p}} \sum_{j_1,j_2=0 \atop j_1+j_2=k}^n 1 \\
& \ll & \frac{n}{2^{2n}} 2^{2 n/p} \sum_{k=n-1}^{2 n} \frac{1}{2^{2 k/p}}\\
& \ll & \frac{n}{2^{2n}}.
\end{eqnarray*}
Altogether we obtain that $$\sum_{j_1,j_2=0 \atop |\bsj| \ge n-1}^n 2^{2|\bsj|} \left\| \sum_{\bsm \in \mathbb{D}_{\bsj}}  \mu_{\bsj,\bsm}^2 \, {\mathbf 1}_{I_{\bsj,\bsm}} \right\|_{L_{p/2}} \ll \frac{n}{2^{2n}}$$ as desired. 
\item $\bsj \in \NN_0^2$, $j_1 \ge n$: According to \textit{(iii)} of Lemma~\ref{lem:HCH} we have
      \begin{eqnarray*}
         \lefteqn{\sum_{j_2=0}^{\infty}\sum_{j_1=n}^{\infty}2^{2|\bsj|}\left\| \sum_{\bsm \in \mathbb{D}_{\bsj}}  \mu_{\bsj,\bsm}^2 \, {\mathbf 1}_{I_{\bsj,\bsm}} \right\|_{L_{p/2}}}\\  & = &   \sum_{j_2=0}^{\infty}\sum_{j_1=n}^{\infty}2^{2|\bsj|}2^{-4(|\bsj|+2)}\left\| \sum_{\bsm \in \mathbb{D}_{\bsj}}  \, {\mathbf 1}_{I_{\bsj,\bsm}} \right\|_{L_{p/2}} \\ &=& \sum_{j_2=0}^{\infty}\sum_{j_1=n}^{\infty}2^{-2|\bsj|-8} \ll \frac{1}{2^{2n}}.
      \end{eqnarray*}
\item $\bsj \in \NN_0^2$, $j_2 \ge n$: Analogous to the case $\bsj \in \NN_0^2$, $j_1 \ge n$.
\item $\bsj=(-1,k)$ with $k \in \NN_0$ and $ 0 \leq k < n$: According to \textit{(iv)} of Lemma~\ref{lem:HCH} we have
     \begin{eqnarray*}
         \lefteqn{\sum_{k=0}^{n-1}2^{2k}\left\| \sum_{\bsm \in \mathbb{D}_{(-1,k)}}  \mu_{(-1,k),\bsm}^2 \, {\mathbf 1}_{I_{(-1,k),\bsm}} \right\|_{L_{p/2}}}\\  & \leq &   \sum_{k=0}^{n-1} 2^{2k}2^{-2(n+k)}\left\| \sum_{\bsm \in \mathbb{D}_{(-1,k)}}  \, {\mathbf 1}_{I_{(-1,k),\bsm}} \right\|_{L_{p/2}} \\ &=& \sum_{k=0}^{n-1}2^{-2n} = \frac{n}{2^{2n}}.
      \end{eqnarray*}
\item $\bsj =(k,-1)$ with $k \in \NN_0$ and $ 0 \leq k < n$: Analogous to the case $\bsj=(-1,k)$ with $k \in \NN_0$ and $ 0 \leq k < n$.
\item $\bsj=(-1,k)$ with $k \in \NN_0$ and $k \geq n$: According to \textit{(v)} of Lemma~\ref{lem:HCH} we have
     \begin{eqnarray*}
         \lefteqn{\sum_{k=n}^{\infty}2^{2k}\left\| \sum_{\bsm \in \mathbb{D}_{(-1,k)}}  \mu_{(-1,k),\bsm}^2 \, {\mathbf 1}_{I_{(-1,k),\bsm}} \right\|_{L_{p/2}}}\\  & = &   \sum_{k=n}^{\infty} 2^{2k}2^{-2(2k+3)}\left\| \sum_{\bsm \in \mathbb{D}_{(-1,k)}}  \, {\mathbf 1}_{I_{(-1,k),\bsm}} \right\|_{L_{p/2}} \\ &=& \sum_{k=n}^{\infty}2^{-2k-6} \ll \frac{1}{2^{2n}}.
      \end{eqnarray*}
\item $\bsj=(k,-1)$ with $k \in \NN_0$ and $k \geq n$: Analogous to the case $\bsj=(-1,k)$ with $k \in \NN_0$ and $k \geq n$.
\item $\bsj=(-1,-1)$: According to \textit{(vi)} of Lemma~\ref{lem:HCH} we have
       \begin{eqnarray}\label{mumineins}
         \left\|   \mu_{(-1,-1),(0,0)}^2 \, {\mathbf 1}_{\left[0,1\right]^2} \right\|_{L_{p/2}}  & = &    \mu_{(-1,-1),(0,0)}^2   \|  \, {\mathbf 1}_{\left[0,1\right]^2} \|_{L_{p/2}} \nonumber \\ &=& \left(\frac{2a_n+4-n}{2^{n+3}}+\frac{1}{2^{2(n+1)}}\right)^2.     \end{eqnarray} 
         Now we assume that $|2 a_n-n| \ll \sqrt{n}$. Then we have
         \[ \left\|   \mu_{(-1,-1),(0,0)}^2 \, {\mathbf 1}_{\left[0,1\right]^2} \right\|_{L_{p/2}}\ll \frac{n}{2^{2n}}. \]
\end{itemize}
Altogether this proves inequality \eqref{eq2do} and therefore also the first point of Theorem~\ref{thm1}. 

It remains to show that the condition on $a_n$ is also necessary. We use again Lemma~\ref{lpi} with $f=D_{N}({\cal R}_{n,\bssigma},\cdot)$ and obtain
\begin{eqnarray*}
L_{p,N}({\cal R}_{n,\bssigma}) & = & \|D_{N}({\cal R}_{n,\bssigma},\cdot)\|_{L_p}\\
& \gg_p & \|S(D_{N}({\cal R}_{n,\bssigma},\cdot))\|_{L_p}\\
& = & \left\| \left( \sum_{\bsj \in \NN_{-1}^2} \sum_{\bsm \in \mathbb{D}_{\bsj}} 2^{2|\bsj|} \, \mu_{\bsj,\bsm}^2 \, {\mathbf 1}_{I_{\bsj,\bsm}} \right)^{1/2} \right\|_{L_p}\\
& = & \left\|\sum_{\bsj \in \NN_{-1}^2} 2^{2|\bsj|} \sum_{\bsm \in \mathbb{D}_{\bsj}}  \mu_{\bsj,\bsm}^2 \, {\mathbf 1}_{I_{\bsj,\bsm}} \right\|_{L_{p/2}}^{1/2}\\
& \gg & \left\|   \mu_{(-1,-1),(0,0)}^2 \, {\mathbf 1}_{\left[0,1\right]^2} \right\|_{L_{p/2}}^{1/2}\\
& = & \left|\frac{2a_n+4-n}{2^{n+3}}+\frac{1}{2^{2(n+1)}}\right|,
\end{eqnarray*}
where the last equality follows from \eqref{mumineins}. From this it is evident that $$L_{p,N}({\cal R}_{n,\bssigma}) \ll_p \frac{\sqrt{\log N}}{N} \asymp \frac{\sqrt{n}}{2^n}$$ implies $|2a_n -n| \ll_p \sqrt{n}$. $\qed$

\section{$L_p$-discrepancy of symmetrized digit shifted Hammersley point sets}\label{sec5}

We define the {\it two-dimensional symmetrized digit shifted Hammersley point set} 
\begin{equation}\label{shHam}
{\cal R}_{n,\bssigma}^{{\rm sym}}={\cal R}_{n,\bssigma} \cup {\cal R}_{n,\bssigma_{\ast}}, 
\end{equation}
where we put $\bssigma_{\ast}=\bssigma \oplus \mathbf1=(\sigma_1 \oplus 1,\sigma_2 \oplus 1, \dots, \sigma_n \oplus 1)$. This set contains $N=2^{n+1}$ points. It is easy to see that ${\cal R}_{n,\bssigma}^{{\rm sym}}$ can also be written as the union of ${\cal R}_{n,\bssigma}$ with the set of points $$\left\{\left(x,1-\frac{1}{2^n}-y\right)\ : \ (x,y) \in {\cal R}_{n,\bssigma}\right\}.$$ Hence in view of in  Davenport's reflection principle the attribute ``symmetrized'' is appropriate. For an example see Figure~\ref{f18}.

\begin{figure}[htp]
     \centering
     {\includegraphics[width=50mm]{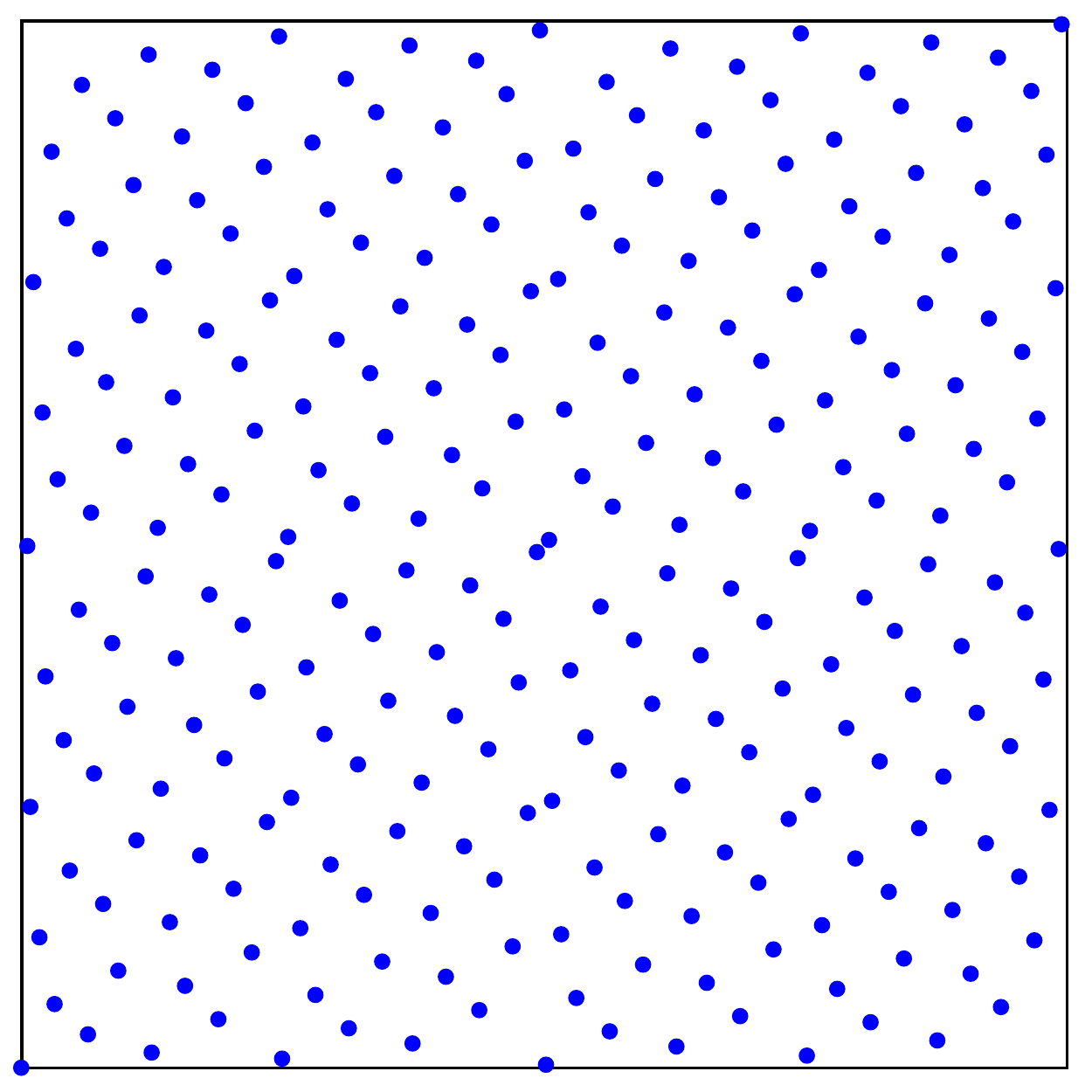}}
     \hspace{.1in}
     {\includegraphics[width=50mm]{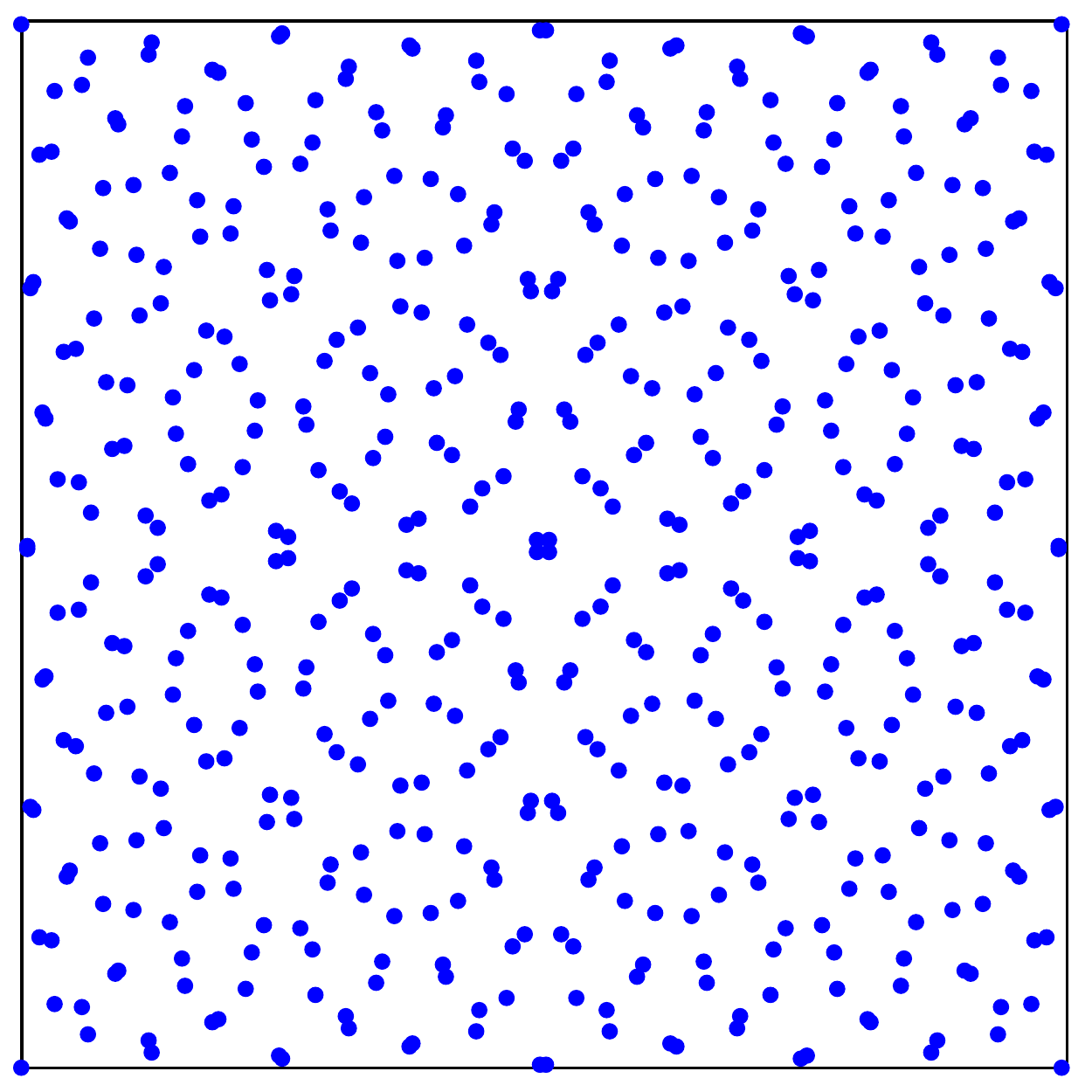}}
     \caption{Two-dimensional Hammersley point set ${\cal R}_{8,{\boldsymbol 0}}$ with $2^8$ elements (left) and symmetrized version ${\cal R}_{8,{\boldsymbol 0}}^{{\rm sym}}$ thereof (right)}
     \label{f18}
\end{figure}

\begin{theorem}\label{thm2}
   Let $p \in (1,\infty)$. Independently of $\bssigma \in \{0,1\}^n$ the two-dimensional symmetrized digit shifted Hammersley point set satisfies
     \[L_{p,N}({\cal R}_{n,\bssigma}^{{\rm sym}})  \ll_p \frac{\sqrt{\log{N}}}{N}.\]
\end{theorem}

For the proof we need upper bounds on the absolute values of the Haar coefficients $\mu_{\bsj,\bsm}^{{\rm sym}}= \langle  D_N({\cal R}_{n,\bssigma}^{{\rm sym}}, \, \cdot \,), h_{\bsj,\bsm} \rangle$ of the local discrepancy of ${\cal R}_{n,\bssigma}^{{\rm sym}}$ which are given in the following:

\begin{lemma} \label{lem:HCSYM} Let $\bsj=(j_1,j_2)\in \NN_{-1}^2$. Then in the case $\bsj\neq (-1,-1)$ we have $$|\mu_{\bsj,\bsm}^{{\rm sym}}| \le |\mu_{\bsj,\bsm}|\ \ \ \mbox{ for all } \ \bsm \in \mathbb{D}_{\bsj}.$$ Hence the results in Lemma~\ref{lem:HCH} apply accordingly also to $|\mu_{\bsj,\bsm}^{{\rm sym}}|$. In the case $\bsj= (-1,-1)$
we have $\mu_{(-1,-1),(0,0)}^{{\rm sym}} = 2^{-(n+1)}+2^{-2(n+1)}$. 
\end{lemma}

\begin{proof} We have
\begin{align*}
    D_N( {\cal R}_{n,\bssigma}^{{\rm sym}},\bst)=&\frac{1}{2^{n+1}}A_N(\left[0,\bst\right),{\cal R}_{n,\bssigma}^{{\rm sym}})-t_1t_2
     \\ =& \frac{1}{2}\left(\frac{1}{2^n}A_N(\left[0,\bst\right),{\cal R}_{n,\bssigma})-t_1t_2+\frac{1}{2^n}A_N(\left[0,\bst\right),{\cal R}_{n,\bssigma_{\ast}})-t_1t_2\right)\\ =& \frac{1}{2}\left( D_N( {\cal R}_{n,\bssigma},\bst)+ D_N( {\cal R}_{n,\bssigma_{\ast}},\bst)\right).
\end{align*}
Regarding the linearity of integration, we obtain \[ \mu_{\bsj,\bsm}^{{\rm sym}}=\frac{1}{2}\left(\mu_{\bsj,\bsm,\bssigma}+\mu_{\bsj,\bsm,\bssigma_{\ast}}\right),\] where here we write $\mu_{\bsj,\bsm,\bssigma}$ for the the Haar coefficients of the local discrepancy of ${\cal R}_{n,\bssigma}$ in order to stress the dependence on the digit shift $\bssigma$ and accordingly for $\mu_{\bsj,\bsm,\bssigma_{\ast}}$. Then the triangle inequality yields
\[ |\mu_{\bsj,\bsm}^{{\rm sym}}|\leq \frac{1}{2}\left(|\mu_{\bsj,\bsm,\bssigma}|+|\mu_{\bsj,\bsm,\bssigma_{\ast}}|\right). \]
We analyze the case $\bsj\neq (-1,-1)$. We note that the identities and upper bounds for $|\mu_{\bsj,\bsm,\bssigma}|$ in Lemma~\ref{lem:HCH} do not depend on the shift $\bssigma$ and therefore we get our desired results in this case directly from this lemma. In case that $\bsj=(-1,-1)$ we observe that the shift $\bssigma_{\ast}$ has $n-a$ zero entries if $\bssigma$ has $a$ zero entries, and thus the result in this case also follows immediately from Lemma~\ref{lem:HCH}. \end{proof}

\noindent{\it Proof of Theorem~\ref{thm2}.} Since the absolute values of the Haar coefficients of $D_N({\cal R}_{n,\bssigma}^{{\rm sym}},\cdot)$ are less than or equal to the absolute values of the Haar coefficients of $D_N({\cal R}_{n,\bssigma},\cdot)$ and since $\mu_{(-1,-1),(0,0)}^{{\rm sym}}$ is of order $2^{-2n}$, the proof of this theorem follows exactly the same lines as the proof of Theorem~\ref{thm1}.$\qed$

Finally we consider a slight variant of the two-dimensional symmetrized shifted Hammersley point set \eqref{shHam}.
Let \[\widetilde{\cal R}_{n,\bssigma}^{{\rm sym}}:={\cal R}_{n,\bssigma} \cup \left\{(x,1-y):(x,y)\in {\cal R}_{n,\bssigma}\right\},\]
where it might happen that two points coincide. The number of elements of $\widetilde{\cal R}_{n,\bssigma}^{{\rm sym}}$, counted by multiplicity, is again $N=2^{n+1}$. It follows from \cite[Theorem~2]{lp}, that the $L_2$-discrepancy of $\widetilde{\cal R}_{n,\bszero}^{{\rm sym}}$ (unshifted) is of optimal order $L_{2,N}(\widetilde{\cal R}_{n,\bszero}^{{\rm sym}}) \ll \sqrt{\log N}/N$. We extend this result to the $L_p$-discrepancy.

\begin{theorem}\label{thm3}
   Let $p \in (1,\infty)$. Independently of $\bssigma \in \{0,1\}^n$ we have
     \[L_{p,N}(\widetilde{{\cal R}}_{n,\bssigma}^{{\rm sym}})  \ll_p \frac{\sqrt{\log{N}}}{N}.\]
\end{theorem}

Theorem~\ref{thm3} follows directly from Theorem~\ref{thm2} in conjunction with the following lemma.

\begin{lemma}
We have 
\begin{equation*} 
|L_{p,N}(\widetilde{\cal R}_{n,\bssigma}^{{\rm sym}})-L_{p,N}({\cal R}_{n,\bssigma}^{{\rm sym}})|\leq \frac{1}{2^{n+1}}=\frac{1}{N}. 
\end{equation*}
\end{lemma}

\begin{proof} At first we note that
\begin{equation}\label{abinq} 
A(\left[\bold 0,\bst\right), \widetilde{\cal R}_{n,\bssigma}^{{\rm sym}})\leq A(\left[\bold 0,\bst\right), {\cal R}_{n,\bssigma}^{{\rm sym}}) \leq A(\left[\bold 0,\bst\right), \widetilde{\cal R}_{n,\bssigma}^{{\rm sym}})+1. 
\end{equation}
For the proof of this claim we consider an arbitrary interval $\left[\bold 0, \bst \right) \subseteq \left[0,1\right]^2$. It is evident that the point set $\widetilde{\cal R}_{n,\bssigma}^{{\rm sym}}$ results from ${\cal R}_{n,\bssigma}^{{\rm sym}}$ if the points in \[\left\{(x,1-1/2^n-y):(x,y)\in {\cal R}_{n,\bssigma}\right\}\] are shifted $1/2^n$ in the positive $y$-direction and the remaining points (which are the elements of ${\cal R}_{n,\bssigma}$) do not move. Since the $y$-coordinates of two distinctive elements in $\left\{(x,1-1/2^n-y):(x,y)\in {\cal R}_{n,\bssigma}\right\}$ differ at least by $1/2^n$, there is at most one element in ${\cal R}_{n,\bssigma}^{{\rm sym}}$ that might leave the interval $\left[\bold 0, \bst \right) $ by shifting these points in the described way, whereas we cannot get additional points in this interval. From these observations the above inequalities \eqref{abinq} are clear. Therefore we obtain
 \[ |D_N({\cal R}_{n,\bssigma}^{{\rm sym}},\bst)-D_N(\widetilde {\cal R}_{n,\bssigma}^{{\rm sym}},\bst)| \leq \frac{1}{2^{n+1}}|A(\left[\bold 0,\bst\right), {\cal R}_{n,\bssigma}^{{\rm sym}})-A(\left[\bold 0,\bst\right), \widetilde{\cal R}_{n,\bssigma}^{{\rm sym}})|\leq \frac{1}{2^{n+1}}. \]
From $||x|-|y||\leq |x-y|$ we get
\[  \left||D_N({\cal R}_{n,\bssigma}^{{\rm sym}},\bst)|-|D_N(\widetilde{\cal R}_{n,\bssigma}^{{\rm sym}},\bst)|\right|\leq \frac {1}{2^{n+1}}.\]
Hence we have
\begin{equation} \label{Inequ1} |D_N({\cal R}_{n,\bssigma}^{{\rm sym}},\bst)| \leq |D_N(\widetilde{\cal R}_{n,\bssigma}^{{\rm sym}},\bst)| +\frac{1}{2^{n+1}}  \end{equation}
and
\begin{equation} \label{Inequ2} |D_N(\widetilde{\cal R}_{n,\bssigma}^{{\rm sym}},\bst)| \leq |D_N( {\cal R}_{n,\bssigma}^{{\rm sym}},\bst)| +\frac{1}{2^{n+1}}.  \end{equation}
Now we take the $L_p$-norm on both sides of inequality~\eqref{Inequ1} and get by regarding the triangle inequality
\begin{align*} 
L_{p,N}({\cal R}_{n,\bssigma}^{{\rm sym}}) & =\left\| D_N( {\cal R}_{n,\bssigma}^{{\rm sym}},\bst) \right\|_{L_{p}} \\
&\leq \|D_N(\widetilde {\cal R}_{n,\bssigma}^{{\rm sym}},\bst)\|_{L_{p}} +\left\|\frac{1}{2^{n+1}}\right\|_{L_p}\\
& =L_{p,N}(\widetilde{\cal R}_{n,\bssigma}^{{\rm sym}})+\frac{1}{2^{n+1}}. \end{align*}
From inequality~\eqref{Inequ2} we derive in an analogue way \[ L_{p,N}(\widetilde{\cal R}_{n,\bssigma}^{{\rm sym}}) \leq L_{p,N}({ \cal R}_{n,\bssigma}^{{\rm sym}})+\frac{1}{2^{n+1}} \] which finally yields the desired result.
\end{proof}

\section{Final remarks}

We have shown two modifications of the classical Hammersley point sets, the digit shifts and the symmetrization, which achieve the optimal order of $L_p$-discrepancy for arbitrary $p \in (1,\infty)$. It should be pointed out that each construction works for all $p \in (1,\infty)$ simultaneously. This is in contrast to the construction of Skriganov~\cite{Skr} where the point sets differ from $p$ to $p$ (but of course the outstanding achievement of Skriganov is that his construction works for arbitrary dimension $s$). Also the point sets constructed in \cite{D14} yield optimal order of $L_p$-discrepancy for all $p \in (1,\infty)$ simultaneously.\\

Finally it should be remarked that recently Goda \cite{goda} presented another modification of two-dimensional Hammersley point sets (in arbitrary base $b$) with optimal order of $L_p$-discrepancy. He considered so-called {\it two-dimensional folded Hammersley point sets} which result from the application of the so-called tent (or bakers) transformation to the elements of the two-dimensional Hammersley point set. In base $2$ this is the point set $${\cal R}_n^{\phi}=\{(\phi(x),\phi(y)) \ : \ (x,y) \in {\cal R}_n\}$$ where $\phi(x)=1-|2x-1|$. Goda showed that $$L_{p,N}({\cal R}_n^{\phi}) \ll_p \frac{\sqrt{\log N}}{N} \ \ \ \mbox{ for all }\ p \in (1,\infty).$$

\noindent{\bf Author's Addresses:}\\

\noindent Aicke Hinrichs, Institut f\"{u}r Analysis, Johannes Kepler Universit\"{a}t Linz, Altenbergerstra{\ss}e 69, A-4040 Linz, Austria. Email: aicke.hinrichs(at)jku.at\\

\noindent Ralph Kritzinger and Friedrich Pillichshammer, Institut f\"{u}r Finanzmathematik, Johannes Kepler Universit\"{a}t Linz, Altenbergerstra{\ss}e 69, A-4040 Linz, Austria. Email: ralph.kritzinger(at)jku.at, friedrich.pillichshammer(at)jku.at

\end{document}